\documentclass[a4paper,11pt,oneside]{amsart}
\usepackage[latin1]{inputenc}
\usepackage[english]{babel}
\usepackage{amsthm}
\usepackage{amssymb}
\usepackage{amsmath}
\usepackage{hyperref}
\usepackage{graphicx}
\usepackage[all]{xy}
\usepackage{ dsfont }

\DeclareMathAlphabet{\mathpzc}{OT1}{pzc}{m}{it}
\newtheoremstyle{note}{11pt}{11pt}{}{}{\bfseries}{.}{.5em}{}
\newtheorem{theo}[equation]{Theorem}
\newtheorem{prop}[equation]{Proposition}

\newtheorem{defin}[equation]{Definition}

\newtheorem{conj}[equation]{Conjecture}
\newtheorem{rem}[equation]{Remark}

\numberwithin{equation}{section}
\newtheorem{lemma}[equation]{Lemma}

\newcommand{\m}[1]{\mathbb{#1}}
\newcommand{\mc}[1]{\mathcal{#1}}
\newcommand{\mr}[1]{\mathrm{#1}}

\newcommand{\Q}{\mathbb{Q}}
\newcommand{\F}{\mathbb{F}}
\newcommand{\T}{\mathbb{T}}

\newcommand{\Z}{\mathbb{Z}}

\newcommand{\R}{\mathbb{R}}
\newcommand{\C}{\mathbb{C}}
\newcommand{\Hbb}{\mathbb{H}}

\newcommand{\Ll}{\mathcal{L}}

\newcommand{\h}{\mathcal{H}}

\newcommand{\W}{\mathcal{W}}

\newcommand{\G}{\Gamma}

\newcommand{\eps}{\varepsilon}

\newcommand{\gl}{\mathrm{GL}}

\newcommand{\Sp}{\mathrm{Sp}}
\newcommand{\GSp}{\mathrm{GSp}}

\newcommand{\boldF}{\mathbf{F}}

\newcommand{\Dfrak}{\mathfrak{D}}

\newcommand{\pfrak}{\mathfrak{p}}

\newcommand{\bfrak}{\mathfrak{b}}

\newcommand{\lgr}{\left\{}
\newcommand{\rgr}{\right\}}
\newcommand{\rra}{\right\rangle}
\newcommand{\lla}{\left\langle}

\title{Derivative of the standard $p$-adic $L$-function associated with a Siegel form}
\author{Giovanni Rosso} \thanks{Herchel Smith Postdoctoral Fellow at the University of Cambridge}

\email{\href{mailto:gr385@cam.ac.uk}{gr385@cam.ac.uk}}
\urladdr{\url{https://sites.google.com/site/gvnros/}}
\address{DPMMS,  Centre for Mathematical Sciences,
Wilberforce Road\\ Cambridge CB3 0WB United Kingdom}

\begin{document}
\maketitle
In this paper we construct a two variables $p$-adic $L$-function for the standard representation associated with a Hida family of parallel weight genus $g$ Siegel forms, using a method previously developed by B\"ocherer--Schmidt in one variable. 
When a form of weight $g+1$ is Steinberg at $p$, a trivial zero appears and, using  the method of Greenberg--Stevens, we calculate the first derivative of this $p$-adic $L$-function and show that it has the form predicted by a conjecture of Greenberg on trivial zeros.
\tableofcontents
\section{Introduction}
Let  $M$ be an irreducible motive, pure of weight $0$ over $\Q$; suppose that $s=0$ is a critical integer \`a la Deligne for $M$ and $L(M,0) \neq 0$.\\ 
We fix a prime number $p$ and let $V$ be the $p$-adic representation associated with $M$. We fix once and for all an isomorphism $\C \cong \C_p$. We suppose that we are given a regular submodule $D$ of the $(\varphi,\Gamma)$-module associated with $V$ \cite[\S 0.2]{BenLinv}. 
Conjecturally, there exists a $p$-adic $L$-function $L_p(V,D,s)$ which interpolates the special values of the $L$-function of $M$ twisted by finite-order characters of $1+p\Z_p$ \cite{PR}, multiplied by a corrective factor (to be thought of as part of the local epsilon factor at $p$) which depends on $D$. In particular, we expect the following interpolation formula at $s=0$:
 \begin{align*}
 L_p(V,D,s) = E(V,D) \frac{L(M,0)}{\Omega(M)},
 \end{align*}
for $\Omega(M)$ a complex period and $E(V,D)$ some Euler type factors which  have to be removed in order to permit $p$-adic interpolation (see \cite[\S 2.3.2]{BenLinv} or \cite[\S 6]{CoaMot}). It may happen that certain of these Euler factors vanish. In this case the connection with what we are interested in, the special values of the $L$-function, is lost. Motivated by the seminal work of Mazur--Tate--Teitelbaum \cite{MTT}, Greenberg, in the ordinary case \cite{TTT}, and Benois \cite{BenLinv} have conjectured the following:
\begin{conj}\label{MainCo}[Trivial zeros conjecture]
Let $e$ be the number of Euler-type factors of $E_p(V,D)$ which vanish. Then the order of zeros at $s=0$ of $L_p(V,D,s)$ is $e$ and
\begin{align}\label{FormMC}
\lim_{s \rightarrow 0} \frac{L_p(V,D,s)}{s^e } = \ell(V,D) E^*(V,D) \frac{L(M,0)}{\Omega(M)}
\end{align}
for $E^*(V,D) $ the non-vanishing factors of $E(V,D)$ and   $\ell(V,D)$ a non-zero number called the $\ell$-invariant as defined in \cite{BenLinv}.
\end{conj}
In this paper we shall study this conjecture for a certain $p$-adic $L$-function associated with Siegel modular forms. Let $f$ be a genus $g$ Siegel modular form of parallel weight $k$, level $\G_0(N)$ and trivial Nebentypus. For each Dirichlet character $\chi$ we can define a $L$-function $\Ll(s, \mr{St}(f),\chi)$, defined as an infinite Euler product of factors of degree $2g+1$, which corresponds, up to some Euler factors at bad primes, to the $L$-function associated with the Standard Galois representation constructed in \cite{ScholzeTor}. 
Its critical values are the integers $1 \leq s \leq k-g$ with $(-1)^{s+g}=\chi(-1)$ and $g+1-k \leq s \leq 0$ with $(-1)^{1+s+g}=\chi(-1)$.

We suppose that $f$ is ordinary for the $U_p$ operator, {\it i.e.} $U_p f=\alpha_p f$ with $\vert \alpha \vert_p = 1$. We know that $f$ can be deformed into a one variable Hida family of ordinary Siegel forms: we have a finite flat $\Z_p[[1+p\Z_p]]$-algebra $\mc O(\m F)$, quotient of the big ordinary Hecke algebra $\m T$, which parametrizes systems of eigenvalues of ordinary Siegel forms of parallel weight. Sometimes in literature this is called a $\mr{GL}_g$-ordinary family. 

We say that a point of $\mr{Spec}(\Z_p[[\Z_p^{\times}]])$ is arithmetic if it corresponds to a character 
\begin{align*}
z \mapsto \eps(z)z^k,
\end{align*}
with $\eps$ a finite order character. We say that a point $x \in \mr{Spec}(\mc O(\m F))(\C_p)$ is arithmetic if the map 
\begin{align*}
w: \mr{Spec}(\mc O(\m F)\otimes \C_p) \rightarrow \mr{Spec}(\Z_p[[X]]\otimes \C_p)
\end{align*}
is \'etale, $w(x)$ is an arithmetic point of type $z \mapsto z^k$ with $k\geq g+1$ and  the corresponding system of eigenvalues in the Hecke algebra is classical. We shall denote by $f_x$ a Siegel eigenform with this system of eigenvalues. 

As $f_x$ is ordinary, there is an obvious choice for the regular sub-module $D$ which appears in the definition of the $p$-adic $L$-function; we shall hence suppress the dependence from $D$ in the notation. This defines in particular a choice of the Satake parameters for $f_x$, as we explain in Section \ref{padicLfunctions}; we shall denote them denote by $\beta_i(f_x)$.
We define: \begin{align*}
E_1(f_x,s)= & \prod_{i=1}^g{(1-\beta_i^{-1}(f_x)p^{s+g-k})},\\
E(f_x,s)  = & E_1(f_x,s)\prod_{i=1}^g\frac{1}{(1-\beta_i(f_x)p^{k-s-g-1})}.
\end{align*}
The first result of the paper is the following:
\begin{theo}
Suppose the ``$p$-adic multiplicity one'' hypothesis of Section \ref{padicLfunctions}. We have an element $L_p(x,s) \in \mr{Frac}(\mc O(\m F) \hat{\otimes} \Z_p[[\Z_p^{\times}]])$ which satisfies, for all point $(x,s)$ arithmetic with $0\leq s \leq k-g-1$, the following interpolation property:
\begin{align*}
L_p(x,s) = E(f_x,s) \frac{\Ll^{(Np)}(\mr{St}(f),s+g+1-k)}{\Omega(f_x,s)},
\end{align*} 
where $\Ll^{(Np)}(\mr{St}(f),s)$ stands for the $L$-function without any Euler factor at $Np$ and $\Omega(f_x,s)$ is a complex period involving the Petersson norm of $f_x$ and powers of $2 \pi i$.
\end{theo}
See Theorem \ref{BigTheo} for more details on the factors that appear, for the interpolation formula at points of type  $(x,\eps,s)$,  for twists by characters $\chi$ of level prime to $p$ and for a way to interpolate $\Ll^{(p)}(\mr{St}(f),s)$. \\

This $p$-adic $L$-function will be constructed using a two-variable generalization of the method of \cite{BS} which expresses $\Ll(s,\mathrm{St}(f))$ as a double Petersson product of $f_x$ with certain Eisenstein series. The $p$-adic interpolation of the Fourier coefficients of these series is the key ingredient for the construction of the $p$-adic $L$-function.

A more general results, always starting from \cite{BS} but using an adelic language, has been given in \cite{LiuSFL} where a $g+1$-variable $p$-adic $L$-function interpolating the other critical values is constructed. 

Remark that the $p$-adic $L$-function depends on a compatible choice of $f_x$ in the family and a different choice would give a different $p$-adic $L$-function. The ``$p$-adic multiplicity one for $\m F$'' is not really necessary and is used to show that we can make such a compatible choice. It could be easily removed as in \cite[\S 7]{LiuSFL} but it helps us simplifying the interpolation formula. 

Possible denominators in the first variable are related to the non-\'etalness of the weight map $w$ and a denominator in the second variable comes from a possible pole of the Kubota--Leopoldt $p$-adic $L$-function which appears in the Fourier coefficients of the $p$-adic Eisenstein series we construct.\\

We study now when this $p$-adic $L$-function presents trivial zeroes.

If $k=g+1$ we can have $E_1(f_x,0)=0$ when $p^{-1}$ appears between  the Satake parameter of $f_x$. We shall say that such an $f_x$ is $\G_0(p)$-Steinberg (see \ref{examples} for a precise definition). If $f_x$ comes from a form of level prime to $p$ we have 
\begin{align*}
E(f_x,s) \Ll^{(p)}(\mathrm{St}(f),s+g+1-k) = & E_1(f_x,s)(1-p^{k-s-g-1})\prod_{i=1}^g{(1-\beta^{-1}_i p^{k-s-g-1})}\Ll(\mathrm{St}(f),s+1+g-k).
\end{align*}
This implies that for $s=k-g-1$ the $p$-adic $L$-function is identically zero. The observation that $\beta_ip^{k-g}$ is an analytic function on $\m F^{\mr {rig}}$ makes one realize that it is possible to exploit the strategy of Greenberg--Stevens to calculate the $p$-adic derivative of $L_p(f,s)$ when $f$ is $\G_0(p)$-Steinberg. 

Indeed, generalizing the paper \cite{RosPB} where the case $g=1$ (the symmetric square of a modular form) has been dealt, one can modify the construction of the Eisenstein family; this allows us to define a second $p$-adic $L$-function $L_p^*(x)$. This new $p$-adic $L$-function satisfies the following equality of locally analytic functions around $f$
\begin{align*}
L_p(x,0)=E(f_x,0)L_p^*(x)
\end{align*}
  and moreover it precisely interpolates the values $\Ll(\mathrm{St}(f),0)$. We can then prove the main theorem of the paper:
  
  \begin{theo}
Let $f$ be a Siegel form of weight $g+1$ and trivial Nebentypus. Suppose that $f$ is $\G_0(p)$-Steinberg and the weight projection $w$ is \'etale at the corresponding point $x$. Then we have:
\begin{align*}
\frac{\textup{d} }{\textup{d}s}{L_p(f,s)}_{\vert_{s=0}}=\ell^{\mathrm{al}}(\mathrm{St}(f))E^*(f) \frac{\Ll^{(Np)}(\mr{St}(f),0)}{\Omega(f)},
\end{align*}
where $\ell^{\mathrm{al}}(\mathrm{St}(f))$ is the Greenberg--Benois $\ell$-invariant as calculated in \cite{RosLinv} and 
\begin{align*}
E^*(f)= & \frac{\prod_{i=2}^g{(1-\beta_i^{-1}p^{-1})}}{\prod_{i=1}^g(1-\beta_ip)}.
\end{align*}
\end{theo}
We conclude the introduction with the remark that this does not prove Conjecture \ref{MainCo} as we do not know if $\ell^{\mathrm{al}}(\mathrm{St}(f))\Ll^{(Np)}(\mr{St}(f),0)$ is not vanishing.

Conjecturally the $\ell$-invariant should not vanish but practically nothing is known at the moment.

The $L$-value can indeed vanish as we are interpolating an imprimitive $L$-function; for certain primes $l\mid N$ we have removed the factor $(1-l^{0})=0$. In this case a study of the second derivative should be necessary;  to the author, this problem seems very hard.\\

\textbf{Acknowledgement} The idea for this work originated during a first visit, founded by the NSF grant  FRG-DMS-0854964, of the author to \'Eric Urban which we thank for the numerous insights and suggestions. We also thank Zheng Liu for sharing with us her preprints and many ideas concerning this work, Jacques Tilouine for  useful conversations on ordinary Siegel forms and Arno Kret for pointing out the paper \cite{ShinPlan}. This work has greatly benefited from an excellent long stay at Columbia Univesity, founded by a FWO travel grant V4.260.14N. 

\section{Notation on Siegel Forms}\label{Sforms}
We now recall the basic theory of (parallel weight) Siegel modular forms. We follow closely the notation of \cite{BS} and we refer to the first section of {\it loc. cit.} for more details. 
Let us denote by $\Hbb_g$ the Siegel space for $\mathrm{GSp}_{2g}$. We have explicitly
\begin{align*}
 \Hbb_{g} = \lgr \left.Z = \left( \begin{array}{cc}
                           Z_1 & Z_2 \\
                           Z_3 & Z_4
                          \end{array}
\right) \right\vert  {Z}^t= Z \mbox{ and } \mathrm{Im}(Z) >0 \rgr.
\end{align*}
It has a natural action of $\mathrm{GSp}^+_{2g}(\R)$ via fractional linear transformations; for any $M = \left( \begin{array}{cc}
                           A & B \\
                           C & D
                          \end{array}
\right)$ in $\mathrm{GSp}^+_{2g}(\R)$ and $Z$ in $\Hbb_{2g}$ we define 
\begin{align*}
 M(Z)= (AZ+B){(CZ+D)}^{-1}.
\end{align*}
For any function $F: \Hbb_g \rightarrow \C$ we define the weight $k$ action
\begin{align*}
F\vert_k M(Z):= F(M(Z)){\mr{det}(CZ+D)}^{-k}\mr{det}(M)^{k/2}.
\end{align*}
Let $\G=\G_0(N)$ be the congruence subgroup of $\mathrm{Sp}_{2g}(\Z)$ of matrices whose lower block $C$ is congruent to $0$ modulo $N$. We consider the space  $M^{(g)}_k(N,\phi)$ of scalar Siegel forms of weight $k$ and Nebentypus $\phi$:
\begin{align*}
 \lgr F: \Hbb_g \rightarrow \C \left\vert  F\vert_k M(Z) = \phi(M)F(Z) \;\:\forall M \in \G, f \textnormal{ holomorphic} \right.\rgr .
\end{align*}
Each $F$ in $M_k^{(g)}(\G,\phi)$ admits a Fourier expansion 
\begin{align*}
 F(Z)= \sum_{T} a(T)e^{2 \pi i \mathrm{tr}(TZ)},
\end{align*}
where $T=\left(  \begin{array}{cc} T_1 & T_2 \\
^t T_2 & T_4                                                                                                                                                            \end{array}
\right)$ ranges over all matrices $T$ positive and semi-defined, with $T_1$, $T_4$ integral matrices and and $T_2$ half-integral.\\

 We have two embeddings (of algebraic groups) of $\mathrm{Sp}_{2g}$ in $\mathrm{Sp}_{4g}$. For all algebra $R$, we have:
\begin{align*}
\mathrm{Sp}_{2g}^{\uparrow}(R) = & \lgr \left( \begin{array}{cccc}
                           a & 0 & b & 0  \\
                           0 & 1_g & 0 & 0 \\
c & 0 & d & 0  \\
                           0 & 0 & 0 & 1_g
                          \end{array}  \right) \left\vert  \left( \begin{array}{cc}
                           a & b \\
                           c & d
                          \end{array}
\right) \right. \in \Sp_{2g}(R)\rgr, \\
\mathrm{Sp}_{2g}^{\downarrow}(R) = & \lgr \left( \begin{array}{cccc}
                           1_g & 0 & 0 & 0  \\
                           0 & a & 0 & b \\
0 & 0 & 1_g & 0  \\
                           0 & c & 0 & d
                          \end{array}  \right) \left\vert  \left( \begin{array}{cc}
                           a & b \\
                           c & d
                          \end{array}
\right) \right.\in \Sp_{2g}(R)\rgr .
\end{align*}
We can embed $\Hbb_g \times \Hbb_g$ in $\Hbb_{2g}$ in the following way:
\begin{align*}
 (z_1, z_4) \mapsto  \left( \begin{array}{cc}
                           z_1 & 0 \\
                           0 & z_4
                          \end{array}
\right).
\end{align*}
This embedding is compatible with the action of $\mathrm{GSp}^+_{2g}(\R)$. For all $\gamma$  we have:
\begin{align*}
 \gamma^{\uparrow} \left( \begin{array}{cc}
                           z_1 & 0 \\
                           0 & z_4
                          \end{array}
\right) = \left( \begin{array}{cc}
                           \gamma(z_1) & 0 \\
                           0 & z_4
                          \end{array}
\right),\\
\gamma^{\downarrow} \left( \begin{array}{cc}
                           z_1 & 0 \\
                           0 &   z_4
                          \end{array}
\right) = \left( \begin{array}{cc}
                           z_1 & 0 \\
                           0 & \gamma (z_4)
                          \end{array}
\right).
\end{align*}
Consequently, evaluation at $z_2=0$ gives us the following map:
\begin{align*}
 M^{(2g)}_k(N,\phi) \hookrightarrow M^{(g)}_k(N,\phi)\otimes_{\C} M^{(g)}_k(N,\phi).
\end{align*}
This also induces a closed embedding of two copies of the Siegel variety of genus $g$ in the Siegel variety of genus $2g$. On points, it corresponds to abelian varieties of dimension $2g$ which decompose as the product of two abelian varieties of dimension $g$.

We define Hecke operators as double cosets operators; let $\gamma$ be an element of $\mathrm{GSp}_{2g}(\Q)$ with positive determinant and $\G_1$, $\G_2$ two congruences subgroups. We write $\G_1 \gamma \G_2 = \sqcup_i \G_1 \gamma_i$ and we define the Hecke operator of weight $k$  
$$ f\vert_k [\G_1 \gamma \G_2]= \sum_{\gamma_i}f\vert_k \gamma_i.$$
We denote by  $\tau_{N}$ the operator $\left[\G_0(N) \left(  \begin{array}{cc}
0 & -1 \\
N & 0                                                                                                                                                      \end{array}
\right)\G_0(N)\right]$. This operator is a generalization of the Atkin--Lehner involution for $\mr{GL}_2$ and is sometime called the Fricke involution. For for any Hecke operator $T$ we define the dual operator $T^*=\tau_N^{-1}T\tau_N$. Once we shall define the Petersson scalar product, $T^*$ will indeed be the dual operator of $T$ for this product at level $\G_0(N)$.

Finally we define the Hecke operator $U_p$ of level $\G_0(p)$ and weight $k$ to be $$ \vert_k {p^{\frac{gk-g(g+1)}{2}}}\left[\G_0(p)\left( \begin{array}{cc}
                           1_g &    \\
                           & p1_{g} 
                          \end{array}
 \right)\G_0(p)\right].$$ If $F(Z)=\sum_{T} a(T)e^{2 \pi i \mathrm{tr}(TZ)}$ we have  from \cite[Proposition 3.5]{HPEL}:
 $$U_pF(Z)=\sum_{T} a(pT)e^{2 \pi i \mathrm{tr}(TZ)}.  $$


\section{Eisenstein series}\label{EisenGSp}
The aim of this section is to recall certain Eisenstein series which will be used to construct the $p$-adic $L$-function as in \cite{BS}. In {\it loc. cit.} the authors consider certain Eisenstein series for $\mathrm{GSp}_{4g}$ whose pullback to two embedded copies of the Siegel variety for $\mathrm{GSp}_{2g}$  is a \textbf{holomorphic} Siegel modular form.
 We now fix a (parallel weight) Siegel modular form $f$ of genus $g$. 
We shall define the standard $L$-function for $f$ $\Ll(\mr{St}(f),s)$ as an Eulerian product. We shall give an integral expression for this $L$-function as a double Petersson product between $f$ and these  Eisenstein series (see Proposition \ref{InteExpr}). Note that for $g=1$ the standard $L$-function of $f$ coincides, up to a twist, with the symmetric square $L$-function of $f$.\\
These Eisenstein series are essentially Ma\ss{}--Shimura derivative of classical Siegel Eisenstein series \cite{Sh7} and the integral formulation of the $L$-function is the classical pullback formula of Garret and Piateski-Shapiro--Rallis. 

The study of the algebraicity of these differential operators has already been exploited in \cite{Har1,Har2} to show that certain values of the standard $L$-function, naturally normalized, are algebraic.

The approach of B\"ocherer--Schmidt consists instead in using a holomorphic differential operator in place of the Maa\ss{}--Shimura operators.

The two develop then a {\it twisting method} which allow to define Eisenstein series whose Fourier coefficients satisfy Kummer's congruences when the character associated with the Eisenstein series varies $p$-adically. 
This is the key for their construction of the one variable (cyclotomic) $p$-adic $L$-function and of our two-variable $p$-adic $L$-function.

 When the character is trivial modulo $p$ there exists a simple relation between the twisted and the not-twisted Eisenstein series \cite[\S 6 Appendix]{BS} which introduces certain Euler factors at $p$ in the interpolation formula for the $p$-adic $L$-function. We shall construct then an {\it improved} $p$-adic $L$-function without these Euler factors (renouncing to the cyclotomic variable) interpolating directly the not-twisted Eisenstein series.

\subsection{Some differential operators}\label{DiffOp}
In this section we recall some holomorphic differential operators used by B\"ocherer--Schmidt. Let $k$ and $s$ be two positive integers; in \cite[(1.15)]{BS}, the authors define a differential operators $\mathring{\Dfrak}_{g,k}^{s} $.  This operator sends a holomorphic function, automorphic of weight $k$ on $\Hbb_{2g}$ to a holomorphic function on $\Hbb_{g} \times \Hbb_{g}$ automorphic of weight $k+s$ in both the first and the second variable. A structure theorem for differential operators on nearly automorphic forms \cite[Proposition 3.4]{ShNHAF} ensure us that this operator is a scalar multiple, depending on $k$ and $s$, of the Ma\ss{}--Shimura operator composed with  the holomorphic projection.
 
For each symmetric matrix $I=\left( \begin{array}{cc}
T_1 & T_2 \\ 
T_2^t & T_4
\end{array}\right)$ of size $2g$, they define a number $\bfrak^{s}_k(I)$  satisfying: 
\begin{align*}
\mathring{\Dfrak}_{g,k}^{s} (e^{\mr{tr}(IZ)})=\bfrak^{s}_k(I)e^{\mr{tr}(T_1z_1+T_4z_4)}.
\end{align*}
If $I$ varies, $\bfrak^{s}_k(I)$ is a polynomial in the entries of $T_i$, homogeneous of degree $sg$. If $\frac{1}{L}T_1$ and $\frac{1}{L}T_4$ are both half-integral then we have the congruence \cite[(1.34)]{BS}:
\begin{align}\label{congru}
4^{gs}\bfrak^{s}_k(I) \equiv 2^{gs}c^s_{g,k}{\mr{det}(2T_2)}^s \bmod L.
\end{align}
This relation (applied for $L=p^n$) tell us that $(U_p^n \otimes U_p^n)\mathring{\Dfrak}_{g,k}^{s}$ is, as a differential operator, modulo $p$ very similar to $(U_p^n \otimes U_p^n) \partial_2^s$, where we denote by $$\partial_2 :=\mr{det} \left(\frac{1}{2}\frac{\partial}{\partial z_{i,j}}\right) \mbox{ for } i=1,\ldots,g \mbox{ and } j=g+1,\ldots,2g.$$ This relation will be the key to interpolate $p$-adically the ordinary projection of the Eisenstein series we are about to introduce.

\subsection{Some Eisenstein series}

We fix a weight $k$, an integer $N$ prime to $p$ and a Nebentypus $\phi$. Let $f$ be an eigenform in $M^{(g)}_k(Np,\phi)$.
Let $R$ be an integer coprime with $N$ and $p$ and $N_1$ a positive integer such that $ N_1 \mid N$.
 We fix a Dirichlet character $\chi$ modulo $N_1Rp$ which we write  as $\chi_1\chi' \eps_1$, with $\chi_1$ defined modulo $N_1$, $\chi'$  primitive modulo $R$ and $\eps_1$ defined modulo $p$. Let us denote by $S$ the product of all prime dividing $Rp$.
Let $t\geq 1$ be an integer and let ${\F^{t+1}\left(z_1,z_4,R^2 N^2 p^{2n},\phi,u\right)}^{(\chi)}$ be the twisted Eisenstein series of \cite[(5.3)]{BS}. We define
\begin{align*}
\h'_{L,\chi}(z_1,z_4,u) := & L(t+1+2s,\phi\chi) \mathring{\Dfrak}_{t+1}^{s}\left({\F^{t+1}\left(\left(\begin{array}{cc}
z_1 & 0 \\
0 & z_4
\end{array} \right),R^2 N^2 p^{2n}S,RN_1p^n,\phi,u\right)}^{(\chi)}\right) \left\vert^z U_{L^2} \right. \left\vert^w U_{L^2}\right. 
\end{align*}
for $s$ a non-negative integer and $p^{n} \mid L$, with $L$ a $p$-power. Note that this Eisenstein series also depends from a real variable $u$. It is a form of level $\G_0(N^2R^2p) \times \G_0(N^2R^2p)$ and weight $k=t+g+s$.\\
We shall sometimes choose $L=1$ and in this case the level is $N^2R^2p^{2n}$.\\
For any prime number $q$ and matrices $I$ as in the previous section, let $B_q(X,I)$ be the polynomial  of degree at most $2g -1$ of \cite[Proposition 5.1]{BS}. 
Let $\eta$ be and $T_2$ a semi-positive definite matrix of size $g \times g$. For a positive integer $t$ we define:
\begin{align*}
 B_{2g}(t) & = {(-1)}^{g(g+t)} \frac{2^{g+ 2gt}}{\G_{2g}\left( g + \frac{1}{2} \right)} \pi^{g+2g^2}, \\
\G_g(s) & = \pi^{\frac{g(g+1)}{4}} \prod_{i=1}^g \G\left(s - \frac{i-1}{2} \right),\\
G_g(T_2,N,\chi) & = \sum_{X \in M_g(\Z) \bmod N} \eta(\mr{det} X) e^{2 \pi i \mr{tr}\left( \frac{1}{L} T_2 X\right)}. 
\end{align*}
When $\chi$ is of conductor $C$ we have \cite[Proposition 6.1]{BS}:
\begin{align*}
G_g(T_2,C,\chi) & =   C^{\frac{1}{2}g(g-1)}\chi^{-1}(\mathrm{det}(T_2)){G(\chi)}^g.
\end{align*}
We deduce easily from \cite[(7.3),(7.14)]{BS} the following theorem :
\begin{theo}\label{FourierEis}
 The  Eisenstein series defined above evaluated at $u=\frac{1}{2}-t$ has the following Fourier development:
\begin{align*}
 \h'_{L,t,\chi}(z_1,z_4) =&  {(Rp^n)}^{\frac{g(g-1)}{2}} B_{2g}(t) {(2 \pi i )}^{sg}{G\left({\chi' \eps}\right)}^g  \sum_{T_1 \geq 0} \sum_{T_4 \geq 0 } \\
 & \left( \sum_I  \bfrak_{g,t+g}^{s}(I)G_g(T_2,N,\chi_1) {\chi'\eps}^{-1}(\mathrm{det}(2T_2)) \sum_{G \in \mathrm{GL}_2(\Z) \setminus \mathbf{D}(I)}(\phi\chi)^2(\mathrm{det}(G)){\vert \mathrm{det}(G) \vert}^{2t-1} \right. \\
& \left. \;\:\;\: L(1-t,\sigma_{-\mathrm{det}(2I)}\phi\chi)\prod_{q \mid\mathrm{det}(2G^{-t}IG^{-1}) } B_q\left(\chi\phi(q)q^{t-g-1},G^{-t}IG^{-1}\right) \right) e^{{2 \pi i} \mr{tr}(T_1 z_1 + T_4 z_4)} ,
\end{align*}
where the sum over $I$ runs along the matrices $\left(  \begin{array}{cc}
L^2 T_1 & T_2 \\
T_2 & L^2 T_4                                                                                                                                                            \end{array}
\right)$  positive definite and with $2 T_2 \in M_{g}(\Z)$, $\chi(M)=0$ if $(M,S)\neq 1 $ and 
\begin{align*}
 \mathbf{D}(I) = \left\{ G \in M_{2g}(\Z) \vert G^{-t}IG^{-1} \textnormal{ is a half-integral symmetric matrix} \right\}.
\end{align*}
\end{theo}
\begin{proof}
The only difference from {\it loc. cit.} is that we do not apply $\left\vert\left(  \begin{array}{cc}
1 & 0 \\
0 & N^2S                                                                                                                                                            \end{array}
\right) \right.$.
\end{proof}
Note that each sum over $I$ is finite because $I$ must have positive determinant. Moreover, we can rewrite it as a sum over $T_2$, with $(2\mr{det}(T_2),N_1Rp)=1$ and $T_1T_4 - T^2_2 >0$.\\

It is proved in \cite[Theorem 8.5]{BS} that (small modifications of) these functions $\h'_{L,\chi}(z,w)$ satisfy Kummer's congruences. The key fact is what  they call the twisting method \cite[(2.18)]{BS}; the Eisenstein series ${\F^{t+1}\left(z,R^2N^2p^{2n},\phi,u \right)}^{(\chi)} $ are obtained  weighting ${\F^{t+1}\left(z,R^2N^2p^{2n},\phi,u \right)}$ with respect to $\chi$ over integral matrices modulo $NRp^n$. To ensure these Kummer's congruences, even when $p$ does not divide the conductor of $\chi$, the authors are forced to consider $\chi$ of level divisible by $p$. Using nothing more than Tamagawa's rationality theorem for $\gl_n$, they find the relation \cite[(7.13')]{BS}:
\begin{align*}
{\F^{t+1}\left(z,R^2N^2p^2, \phi,u\right)}^{(\chi)} = {\F^{t+1}\left(z,R^2N^2p^2, \phi,u\right)}^{(\chi_1\chi')}\left\vert \left( \sum_{i=0}^g{(-1)}^i p^{\frac{i(i-1)}{2}} p^{-ig}\sum_j \left(  \begin{array}{cc}
1 & S(g_{i,j}) \\
0 & 1                                                                                                                                                      \end{array}
\right)  \right) \right. 
\end{align*}
where $g_{i,j}$ runs along the representative of the double cosets in \cite[(2.38)]{BS} and $S(X) $ is the $2g \times 2g$ antidiagonal (in blocks) matrix with $X$ and $X^t$ on the antidiagonal.\\
Consequently, the Eisenstein series we want to interpolate to construct the improved $p$-adic $L$-function is the one appearing inside formula (2.25') of \cite{BS}, namely:
\begin{align*}
{\h'}^*_{L,t,\chi'}(z_1,z_4) := & L(t+1+2s,\phi\chi) {\mathring{\Dfrak}_{t+1}^{s}\left({\F^{t+1}\left(\left(\begin{array}{cc}
z_1 & 0 \\
0 & z_4
\end{array} \right),R^2 N^2 p^2,\phi,u\right)}^{(\chi_1\chi')}\right)}_{\vert_{u=\frac{1}{2}-t}}. 
\end{align*}
In what follows, we shall be interested only in the case where  $t=k - g$.\\
Let $f \in M^{(g)}_k(\G_0(N),\phi)$ as before and $\T_N$ be the abstract Hecke algebra for $\G_0(N)$. We can decompose $\T_N$ as a restricted product $\bigotimes'_q \T_{N,q}$. We suppose that $f$ is an eigenform for $\T_N$. If $q \nmid N$, let us denote by $\alpha_{q,1}^{\pm 1},\ldots, \alpha_{q,g}^{\pm}$ the Satake parameters associated with $f$; they are well defined up to the action of the Weyl group of $\GSp_{2g}$. If $q \mid N$, we denote by $\alpha_{q,1},\ldots, \alpha_{q,g}$ the Satake parameters associated with $f$, well defined up to permutation. 

For each Dirichlet character $\chi$ we define the standard $L$-function $\Ll(\mathrm{St}(f),\chi,s)$ associated with $f$ and $\chi$  as the infinite product:
\begin{align*}
\Ll(\mathrm{St}(f),\chi,s)=& \prod_{q}{D_q(\chi(q)q^{-s})}^{-1},
\end{align*}
where \begin{align*}
D_q(T)=& (1-\phi(q)T) \prod_{i=1}^g(1-\phi(q)\alpha_{q,i}^{- 1}T)(1-\phi(q)\alpha_{q,i}T)  \mbox{ if } q \nmid N,\\
D_q(T)=& \prod_{i=1}^g(1-\alpha_{q,i}T)  \mbox{ if } q \mid N.
\end{align*}
This $L$-function differs from the (conjectural) motivic/automorphic $L$-function $L(\mathrm{St}(f)\otimes \chi\phi,s)$ by a finite number of Euler factors at primes dividing $N$. See \cite[Table A.10]{RSLocal} for these local factors when $g=2$. We shall also write 
\begin{align*}
\Ll^{(N_1)}(\mathrm{St}(f),\chi,s)=& \prod_{q \nmid N_1}{D_q(\chi(q)q^{-s})}^{-1}.
\end{align*}

In what follows, for two forms $f$ and $g$ of level $\G$ and weight $k$, we shall denote by $\lla f,g \rra$ the normalized Petersson product:
\begin{align*}
{\lla f,g \rra}_{\G}:=\int_{\Hbb_g/\G} f(z) \overline{g(-\overline{z})} \mr{det}(\mr{Im}(z))^k.
\end{align*}

We conclude with the integral formulation of $\Ll(\mathrm{St}(f),\chi,s)$ \cite[Theorem 3.1, Proposition 7.1 case (7.13)]{BS}. 
\begin{prop}\label{InteExpr}
 Let $f$ be a form of weight $k$, Nebentypus $\phi$. We  put $t+s = k-g$ and $\h'= \h'_{1,t,\chi}(z,w)$; we have
\begin{align*}
\lla f(w)\left\vert \left(  \begin{array}{cc}
0 & -1 \\
N^2p^{2n} R^2 & 0                                                                                                                                                      \end{array}
\right)   \right.,  \h'\rra  = & \frac{\Omega_{k,s}({\frac{1}{2}-t}) p_{\frac{1}{2}-t}(t+g)}{\chi(-1) d_{\frac{1}{2}-t}(t+g)} {(R^2N^2 p^{2n}) }^{\frac{g(g+1)-gk}{2}} {(N_1Rp^n)}^{g(k-t)}  \\
& \times \chi(-1)^g{(-1)^{gk}} \Ll^{(N_1)}(1-t,\mathrm{St}(f),\chi^{-1})    f(z)\vert U_{N^2/N_1^2} 
\end{align*}
where the Petersson norm is for forms of level $\G_0({N^2p^{2n} R^2})$ and $\frac{ p_{\frac{1}{2}-t}(t+g)}{d_{\frac{1}{2}-t}(t+g)}\Omega_{k,s}\left( \frac{1}{2} -t\right)$ is an explicit products of $\G$-functions (see Proof of Theorem \ref{BigTheo}).

If $\chi$ is trivial modulo $p$, we let ${\h'}^*={\h'}_{1,t,\chi'}^*(z,w)$ and we have:
\begin{align*}
\lla f(w)\left\vert \left(  \begin{array}{cc}
0 & -1 \\
N^2p^2 R^2 & 0                                                                                                                                                      \end{array}
\right)   \right.,  {\h'}^*\rra
p^{-\frac{g(g+1)}{2}} \prod_{i=1}^{g}(1-\beta_i^{-1}(f_x)\chi^{-1}(p)p^{-t})= \lla f(w)\left\vert \left(  \begin{array}{cc}
0 & -1 \\
N^2p^{2n} R^2 & 0                                                                                                                                                      \end{array}
\right)   \right.,  \h'\rra.
\end{align*}

\end{prop}
\begin{proof}
With the notation of \cite[Theorem 3.1]{BS} we have $M=R^2 N^2 p^{2n}$ and $N=N_1Rp^n$. We  have that $\frac{d_{\frac{1}{2}-t}(t+g)}{p_{\frac{1}{2}-t}(t+g)}\h'$ is the holomorphic projection of the Eisenstein series of  \cite[Theorem 3.1]{BS} (see \cite[(1.30),(2.1),(2.25)]{BS}) and the factor ${(-1)}^{gk}$ comes from the Atkin--Lehner involution.

The second formula is \cite[\S 3 Appendix]{BS}.
\end{proof}
\begin{rem}
This is essentially a reformulation of the classical pull-back formula of Garrett and Piateski-Shapiro--Rallis \cite[Part A]{GPSR}. 
\end{rem}
If $U_N f \neq 0$ we can then take $N_1=1$ in which case we do not remove any extra factor.
\section{Families of Eisenstein series}
\subsection{Families of ordinary Siegel forms}
We recall Hida's theory for parallel weight Siegel forms, following \cite[\S 4]{PilHida}. 
For $i=0,\ldots,g$ we define the Hecke operators:
$$U_{p,i}=\left[\G\left( \begin{array}{cccc}
                           1_i &  & &  \\
                           & p^{-1}1_{g-i} & &  \\
                            &  &  1_{i}&   \\
                            &  & & p1_{g-i} 
                          \end{array}
\right)\G\right],$$ and $U_p={p^{\frac{gk-g(g+1)}{2}}}\left[\G\left( \begin{array}{cc}
                           1_g &  \\
                           & p1_{g} 
                          \end{array}
\right)\G\right]$. Note that $U_p^2=p^{{gk-g(g+1)}}U_{p,g}$ This normalization is optimal if one wants the $U_p$ operator to preserve both $p$-integrality and non-vanishing modulo $p$. 
 We let ${V}_{l,\infty}$ be the set of $p$-adic modular forms of tame level $N$ with $\Z/p^l\Z$-coefficients (of any weight); it is equipped with an action of $\mr{GL}_g(\Z_p)$. We define $V_{\infty}=(\varinjlim_l {V}_{l,\infty})^{\mr{SL}_g(\Z_p)}$, which should be thought as a space of $p$-adic Siegel forms with $\Q_p/\Z_p$ coefficients. 
 We have a $\m U_p$-operator acting on $V_{\infty}$ and we define $\m V=\mr{Hom}_{\Z_p}( e_{\mr{GL}_g} V_{\infty}, \Q_p/\Z_p)$ for $e_{\mr{GL}_g}=\lim_n \m U_p^{n!}$. This is the ordinary projector ``adapted to parallel weight'' of \cite[D\'efinition 5.1]{PilHida}.
 
 We let $\Lambda=\Z_p[[\Z_p^{\times}]]$, it has a natural $\Z_p$-linear action on $V_{\infty}$:
 $z \in \Z_p^{\times}$ acts via any matrix $g \in \mr{GL}_g(\Z_p)$ of determinant $z$. We shall write $z \in \Z_p^{\times}$ as $\omega(z)\lla z \rra$, being $\omega$ the Teichm\"uller character. 
 
  Hence we can consider $V_{\infty}$ and $\m V$ as $\Lambda$-modules and we define 
 \begin{align*}
 M_{\infty} := & \mr{Hom}_{\Lambda}(\mr{Hom}_{\Z_p}(  V_{\infty}, \Q_p/\Z_p),\Lambda),\\
 \m M := & \mr{Hom}_{\Lambda}(\m V, \Lambda).
 \end{align*}
 
 The module $\m M$ is the (free of finite type) module of $\Lambda$-adic ordinary forms. As shown in \cite{LiuSFL}, each element of $M_{\infty}$ or $\m M$ admits a $q$-expansion with $\Lambda$ coefficients; indeed, each element in $\m V$ has a $q$-expansion in $\Q_p/\Z_p$ and each $T$ defines a map $\tilde{\lambda}_T:\m V \rightarrow \Q_p/\Z_p$. Applying duality twice to $\tilde{\lambda}_T$ we obtain a map $\lambda_T:\m M \rightarrow \Lambda$ which defines the $q^T$-th coefficient of an element of $\m M$ and commute with specialization at any  point of $\Lambda$. \\
 
We have an action of the Hecke algebra $\m T$ on $ \m M$ which makes $\m M$ a finite module over $\m T$. Moreover $\m T$ is finite over $\Lambda$. We shall call an irreducible component $\m F$ of  $\mr{Spec}(\m T)$ a Hida family of Siegel eigensystem. To each $\m F$ we can associate the corresponding eigenspace on $\m M$ consisting of $\mc O (\m F)$-adic $q$-expansions.\\

 The same theory can be developed for Siegel form with Nebentypus outside $p$ \cite[Th\'eor\`eme 7.2]{PilHida}. In what follows we shall, if not clear from the context,  emphasize the dependence from the tame level $N$ by writing $\m M_N$ and $\m T_N$.\\

We conclude this section recalling this important theorem, originally due to Hida, which in this form is due Pilloni:
\begin{theo}
Let $k$ be an integer, $i$ an integer, $0 \leq i < p-1$ and $P_k$ the kernel of the map $[i,k]:\Lambda \rightarrow \Z_p$ sending $z \in \Z_p^{\times}$ to $\omega^i(z) \lla z \rra^k \in \Z_p$. For $k$ big enough we have 
\begin{align*}
\m M \otimes \Lambda/P_k \cong M_k^{(g)}(\mr{Iw}_N,\chi \omega^{-i})^{\mr{ord}},
\end{align*}
where the subscript ord refers to ordinarity for $U_p$ and $\mr{Iw}_N$ the intersection of $\G_0(N)$ with the $p$-Iwahori subgroup of $\mr{GSp}_{2g}$ of matrices which are upper triangular modulo $p$.
\end{theo}

\begin{rem}
Conjecturally, the optimal $k$ for the theorem to hold is $g+1$ and this is proven for $g=1$. For $g=2$, the best bound is $4$ \cite{PilDuke}.
\end{rem}

We want to point out the Hecke action on the space of Siegel forms in \cite{PilHida} and in this paper differ by a character. Namely, using the complex uniformization of the Siegel variety, the action in {\it loc. cit.} would be given by 
\begin{align*}
F\vert_k M (Z) =F(M(Z)){\mr{det}(CZ+D)}^{-k}\mr{det}(M)^{k}.
\end{align*}

\subsection{The Eisenstein family}
We now  use the above defined Eisenstein series to give examples of $p$-adic families. More precisely, we shall construct a two-variable measure (which will be used for the two variables $p$-adic $L$-function) and a one variable measure (which will be used to construct the improved one variable $L$-function).\\
Let us fix $\chi=\chi_1\chi'\eps_1$ as before. We suppose  $\chi$ even. We recall the Kubota-Leopoldt $p$-adic $L$-function:
\begin{theo}
 Let $\eta$ be a even Dirichlet character. There exists a $p$-adic $L$-function $L_p(\kappa,\eta) $ satisfying the following interpolation formula for any integer $t\geq 1$ and finite-order character $\eps$ of $1+p\Z_p$:
\begin{align*}
 L_p(\eps(u)[t],\eta) = (1-(\eps\omega^{-t}\eta)_0(p))L(1-t,\eps\omega^{-t}\eta),
\end{align*}
where $\eta_0$ stands for the primitive character associated with $\eta$. If $\eta$ is not trivial then $ L_p([t],\eta)$ is holomorphic. Otherwise, it has a simple pole at $[0]$.
\end{theo}

We can consequently define $p$-adic analytic functions interpolating the Fourier coefficients of the Eisenstein series defined in the previous section; for any $z$ in  $\Z_p^*$, we define $l_z=\frac{\log_p(z)}{\log_p(u)}$. For $T_1$ and $T_4$ two positive semi-definite matrices we define the function:
\begin{align*}
 a_{T_1,T_4,L}(\kappa,\kappa')= & \left( \sum_I \kappa{\kappa'}^{-1}[-g](u_{\mr{det}(2 T_2)}) {\chi'}^{-1}(\mr{det}(2T_2)) \times \right. \\
 & \left. \times \sum_{G \in \mathrm{GL}_{2g}(\Z) \setminus \mathbf{D}(I)}(\phi\chi_1\chi')^2(\mathrm{det}(G)){\vert \mathrm{det}(G) \vert}^{-1}{\kappa'}^2(u^{l_{\vert \mathrm{det}(G) \vert}}) \right. \\
& \left. \;\:\;\: L_p(\kappa',\sigma_{-\mathrm{det}(2I)}\phi\chi)\prod_{q \mid\mathrm{det}(2G^{-t}IG^{-1}) } B_q\left(\phi(q)\kappa'(u^{l_q})q^{-g-1},G^{-t}IG^{-1}\right) \right)
\end{align*}
for $I$ as in Theorem \ref{FourierEis}. (Here $\kappa$ is the weight variable and $\kappa'$ is the variable for $t$. Recall that $s=k-g-t$.)\\
Note that the terms for $I$ with $(\mr{det}(T_2),p) \neq 1$ are always $0$.\\
We recall that if $p^{-j}T_1$ and $p^{-j}T_4$ are half integral we have \cite[(1.21, 1.34)]{BS}:
\begin{align*}
 4^{gs}\bfrak_{g,t+1}^{s}(I) \equiv {(-1)}^{s}2^{gs}  c^{s}_{g,t+1} {\mr{det}(2T_2)}^{s} \bmod p^{j},
\end{align*}
for $s=k-t-g$. Consequently, if we define 
\begin{align*}
 \h_L(\kappa,\kappa') = \sum_{T_1 \geq 0} \sum_{T_4 \geq 0 }  a_{T_1,T_4,L}(\kappa,\kappa') q_1^{T_1}q^{T_4}_2
\end{align*}
 from the congruence (\ref{congru}) we have:
\begin{align}\label{congru2}
{(-1)}^{s} 2^{s} c^{s}_{g,t+g} A  \h_{p^j}([k],\eps[t]) \equiv  \h'_{p^j,t,\chi\eps\omega^{-t}}(z,w) \bmod p^j,
\end{align}
with 
\begin{align*}
 A = A(k,t,\eps)=  {(Rp^n)}^{\frac{g(g-1)}{2}} B_{2g}(t) {(2 \pi i )}^{sg}{G\left({\chi\eps\omega^{-t}}\right)}^g.
\end{align*}

We have the following lemma:
\begin{lemma}\label{lemmaprojinfty}
 There exists a projector \begin{align*}
e_{g,g}^{\mr{ord}} : {M}_{\infty}\otimes{M}_{\infty} \rightarrow \m{M} \otimes\m{M}
                          \end{align*}
                          such that, for $i$, $j$ big enough the following holds:
                          \begin{align*}
                          {(U^{\otimes^2}_{p})}^{-2i}e_{g,g}^{\mr{ord}}{(U^{\otimes^2}_{p})}^{2i} F(\kappa)= {(U^{\otimes^2}_{p})}^{-2j}e_{g,g}^{\mr{ord}}{(U^{\otimes^2}_{p})}^{2j} F(\kappa).
                          \end{align*}
\end{lemma}
\begin{proof}
We define $e_{g,g}^{\mr{ord}}:= e_{\mr{GL}_g}^{\otimes^2}$. The rest is an immediate consequence of the fact that $U_p$ and $e_{\mr{GL}_g}$ commutes and that $U_p$ is invertible on the ordinary part.
\end{proof}

We shall now construct  the improved Eisenstein family.  We let  $\chi=\chi_1\chi'$ be a Dirichlet character modulo $N_1R$ and (as before) we see $\phi$ as a Dirichlet character modulo $Np$. We define
\begin{align*}
 a^*_{T_1,T_4}(\kappa)= & \left( \sum_I {(\chi'\chi_1)}^{-1}(2T_2) \sum_{G \in \mathrm{GL}_2(\Z) \setminus \mathbf{D}(I)}(\phi\chi'\chi_1)^2(\mathrm{det}(G)){\vert \mathrm{det}(G) \vert}^{-1-2g}{\kappa}(u^{l_{2\vert \mathrm{det}(G) \vert}}) \right. \times \\
& \left. \;\:\;\: \times L_p(\kappa[-g],\sigma_{-\mathrm{det}(2I)}\phi\chi'\chi_1)\prod_{q \mid\mathrm{det}(2G^{-t}IG^{-1}) } B_q\left(\phi\chi'\chi_1(q)\kappa(u^{l_q})q^{-1},G^{-t}IG^{-1}\right) \right).
\end{align*}
Note that the terms for $I$ with $(\mr{det}(T_2),p) \neq 1$ are not necessarily $0$ contrary to the previous construction. We now construct another $p$-adic family of Eisenstein series:
\begin{align*}
  \h^*(\kappa)=\sum_{T_1 \geq 0} \sum_{T_4 \geq 0 }  a^*_{T_1,T_4}(\kappa[1-k_0]) q_1^{T_1}q^{T_2}_2.
\end{align*}
We can construct a two-variable family of Eisenstein series, generalizing \cite[Theorem 8.6]{BS}:
\begin{prop}
 We have a $p$-adic measure $\h(\kappa,\kappa') $ on $\Z_p^{\times} \times \Z_p^{\times}$  with values in $ \m{M} \otimes\m{M}$ such that for all $k,t,\eps$ arithmetic, with $\eps$ of conductor $p^n$, we have
\begin{align*}
 \h([k],\eps[t]) = \frac{1}{{(-1)}^{s} c^{s}_{g,t+g} A }  e_{g,g}^{\mr{ord}}   \h'_{1,t,\chi\eps\omega^{-t}}(z,w),
\end{align*}
for any $i\geq 2n$.

We have a one variable measure $\h^*(\kappa)$ on $\Z_p^{\times}$ with values in $\m{M} \otimes\m{M}$ such that
\begin{align*}
 \h^*([k]) = {A^*}^{-1}e_{g,g}^{\mr{ord}}{\h'}^*_{1,\chi_1\chi'}(z,w)\vert_{u=\frac{1}{2}},
\end{align*}
for
\begin{align*}
 A^{*}=  A^*(k)=  {R}^{\frac{g(g-1)}{2}} B_{2g}(k-g) {(2 \pi i )}^{sg}{G\left({\chi'}\right)}^g. 
\end{align*}
\end{prop}
\begin{proof}
 Note that from its own definition we have ${(U^{\otimes^2}_{p})}^{2j}\h'_{1,\chi}(z,w) =\h'_{p^j,\chi}(z,w)$. We define 
\begin{align}\label{EisMeasure}
 \h(\kappa,\kappa')= e_{g,g}^{\mr{ord}}\h_{1}(\kappa,\kappa').
\end{align}
The key congruence (\ref{congru2})  give us %
\begin{align*}
A {(-1)}^s 4^{s} c^{s}_{g,t+g} e_{g,g}^{\mr{ord}} \h_{1}([k],\eps[t]) = & A {(-1)}^s 4^{s} c^{s}_{g,t+g} {U^{\otimes^2}_{p}}^{-2j} e_{g,g}^{\mr{ord}} \h_{p^j}([k],\eps[t])  \\
\equiv & 4^{s}{U^{\otimes^2}_{p}}^{-2j}e_{g,g}^{\mr{ord}} \h'_{p^j,t,\chi\eps\omega^{-t}}(z,w) \\
  = & 4^{s}e_{g,g}^{\mr{ord}} \h'_{1,t,\chi\eps\omega^{-t}}(z,w) 
\end{align*}
as $U_p$ acts on the ordinary part with norm $1$ hence it preserves the $q$-expansion norm (which induces the sup-norm on the ordinary locus). Taking the limit over $j$ gives the desired result. Similarly we define
\begin{align}\label{EisMeasure*}
 \h^*(\kappa)= e_{g,g}^{\mr{ord}}\h^*(\kappa,\kappa').
\end{align}

\end{proof}
We want to explain briefly why the construction above works in the ordinary setting and not in the finite slope one.\\
 It is slightly complicated to explicitly  calculate the polynomial $\bfrak_{t+1}^{s}(I)$ and in particular to show that they vary $p$-adically with $s$. But we know that $\mathring{\Dfrak}_{g,k-2s}^{s}$ is an homogeneous polynomial in $\frac{\partial}{\partial z_{i,j}}$ of degree $gs$.  Consider the embedding $\Hbb_g \times \Hbb_g \hookrightarrow \Hbb_{2g}$; we have a single monomial of $\mathring{\Dfrak_{g,k}^{s}}$ which does belong to the normal bundle, namely $ c^{s}_{g,t+1} \partial_2^{s}$, and this is the term that does not involve a partial derivative $\frac{\partial}{\partial z_{i,j}}$ in a variable $z_{i,j}$  of  $\Hbb_g^2$. Consequently, in $\bfrak_{t+1}^{s}(I)$ there is a single monomial without $T_1$ and $T_4$.\\
 When  the entries on the diagonal of $I$ are divisible by $p^i$, $\bfrak_{t+1}^{s}(I)$ reduces to $  c^{s}_{t+1}$ modulo $p^i$. Hence applying $U^{\otimes^2}_{p}$ many times we approximate $\mathring{\Dfrak}_{g,k}^{s}$ by ${\partial}^s$ (multiplied by a constant). The more times we apply $U^{\otimes^2}_{p}$,  the better we can approximate $p$-adically $\mathring{\Dfrak_l^{s}}$ by $\partial^s_2$. At the limit, we obtain equality.

 The same method does not work for finite slope forms, as the finite slope projector is very different for Hida's ordinary projector. To construct $p$-adic $L$-functions for families of finite slope forms one should then be able to estimate the overconvergent norm of these derivatives and show that they satisfy certain distribution relations. The interested reader is referred to \cite{LiuNHF} where the $p$-adic avatars of Maa\ss{}--Shimura operators are studied.

\section{$p$-adic $L$-functions}\label{padicLfunctions}
We now construct two $p$-adic $L$-functions using the above Eisenstein measures: the two-variable one and the improved. We fix a tame level $N$ and two characters: $\chi_1$ modulo $N_1$ for $N_1 \mid N$, and $\chi' $ modulo $R$, with $(R,Np)=1$. Let $R_0$ be the product of all the primes dividing $R$, $S=R_0p$ and $\chi=\chi'\chi_1$. Consider an irreducible component $\mathbb{F}$ of the ordinary Hecke algebra $\T_{NR_0}$; we suppose that all the classical specializations of $\mathbb{F}$ have level $\G_0(p)$ at $p$. 

We make the following hypothesis 
\begin{center}
\textbf{p-adic Mult. one} The generalized Hecke eigenspace on $\m{M}\otimes_{\Lambda}\mr{Frac}(\Lambda)$ associated with $\mathbb{F}$ is one-dimensional.
\end{center}
This hypothesis is not really necessary but simplify the definition and the evaluation of the $p$-adic $L$-function. The reader interested in removing ths hypothesis is referred to \cite[\S 7]{LiuSFL}. 

Let us decompose
\begin{align*}
\T^{\mr{ord}}\otimes_{\Lambda}\mr{Frac}(\Lambda) = \mathbb{I}_{\F}\oplus\mathbb{B}
\end{align*}
and let $1_{\F}$ be the corresponding idempotent. Let us denote by $\boldF$ a Siegel form in  $1_{\F}\mathcal{M}^{\mr{ord}}\otimes_{\Lambda}\mr{Frac}(\Lambda)$.

We define a twisted trace operator, following ideas of Perrin-Riou and Hida \cite[1.VI]{H1bis}. Let $N $ and $ L$ be two integers, we define the operator $[L/N]$ 
to be 
$$
\begin{array}{ccccc}
[N/L]& : & M^{(g)}_k(\G_0(L),\phi) & \longrightarrow & M^{(g)}_k(\tiny{\left( \begin{array}{cc}
                           1_g & 0 \\
                           0 & L/N 1_g
                          \end{array}
\right)} \G_0(L)\tiny{\left( \begin{array}{cc}
                           1_g & 0 \\
                           0 & N/L 1_g
                          \end{array}
\right)} ,\phi)  \\
& & f & \mapsto &  f \vert_k 
\left( \begin{array}{cc}
                           1_g & 0 \\
                           0 & N/L 1_g
                          \end{array}
\right) 
\end{array}
$$

Let now $N \mid L$ and define the twisted trace 
\begin{align*}
T_{L/N}:= \mr{Tr}_{L/N} \circ [N/L],
\end{align*}
for $\mr{Tr}_{L/N}$ the trace from $M^{(g)}_k(\tiny{\left( \begin{array}{cc}
                           1_g & 0 \\
                           0 & L/N 1_g
                          \end{array}
\right)} \G_0(L)\tiny{\left( \begin{array}{cc}
                           1_g & 0 \\
                           0 & N/L 1_g
                          \end{array}
\right)} ,\phi)$ to $M^{(g)}_k(\G_0(N),\phi)$. If $L$ is coprime with $p$ than this operator can be $p$-adically interpolated over the weights space. Let $R_0$ be the product of all prime factors of $R$. 
We hence define $L_p(\kappa,\kappa')$ to be such that 
\begin{align*}
1_{\F}\otimes 1_{\F}(\mr{Id}{\otimes}T_{N^2R^2/NR_0}\h(\kappa,\kappa'))=L_p(\kappa,\kappa')\boldF \otimes \boldF.
\end{align*}
\begin{rem}
It is clear that this $p$-adic $L$-function depends from the choice of $\boldF$. We do not have a duality between modular form and Hecke algebra so there is no clear choice for such $\boldF$. In contrast, for $g=1$, there is a duality given by the first Fourier coefficient $a(1,f\vert T)$, as $a(1,T_nf)=a(n,f)$.
\end{rem}

Similarly we have a one-variable $p$-adic $L$-function $L_p^*(\kappa)$ defined as
\begin{align*}
1_{\F}\otimes 1_{\F}(\mr{Id}{\otimes}T_{N^2R^2/NS'}\h^*(\kappa))=L^*_p(\kappa)\boldF \otimes \boldF.
\end{align*}

For a form $f_x$ in the family $\boldF$ we denote by $\beta_i(f_x)$ the Satake parameters for $f_x$. We have $\prod_i \beta_i(f_x) = p^{\frac{g(g+1)}{2}-gk}U_{p}^2$. As $f_x$ is ordinary we know that it is of finite slope for the operators $U_{p,i}$. In particular, we have that $\beta_i p^{k-i}$ defines analytic functions $\mc O(\m F)[1/p]$. As this will not change the interpolation properties of the $p$-adic $L$-function, we  can and shall choose the indexing of the $\beta_i$'s such that $\beta_i p^{k-i}$ is a $p$-adic unit. This amounts to say that $f_x$ is ordinary for all the $U_{p,i}$. We shall denote by $\mathds{B}_i$ the corresponding unit in $\mc O(\m F)$.

We define \begin{align*}
E_1(f_x,\eps,t)= & \prod_{i=1}^g{(1-(\chi\eps\omega^{-t})(p)\beta_i^{-1}p^{-t})},\\
E(f_x,\eps,t)  = & E_1(f_x,\eps,t)\prod_{i=1}^g\frac{1}{(1-(\chi^{-1}\eps^{-1}\omega^{t})(p)\beta_ip^{t-1})}.
\end{align*}
We shall write $w:\m F^{\mr{rig}} \rightarrow \W:=\mr{Spf}(\Lambda)^{\mr{rig}} $. We shall say that a point $x \in \m F (\Q_p)$ is \'etale if $w$ is \'etale at $x$.
The main theorem of the section is the following:
\begin{theo}\label{BigTheo}
For points $(x,\kappa')$ of type $(k,t,\eps)$ such that $x$ is \'etale, $1 \leq t \leq k-g$ we have the following interpolation formula:
\begin{align*}
L_p(x,\eps[t])= & {2^{1-g-gs}\alpha_{N^2/N_1^2}(f_x)R^{(1-t)g+\frac{g(g+1)}{2}}N^{g(g+1-k)}N_1^{g(k-t)} }\chi(-1)^g{(-1)}^{\frac{gk}{2}+gs+tg}\\
 & \times \frac{\prod_{i=1}^g(s+g-i)!}{{(2 \pi i)}^{sg} \pi^{\frac{g(g+1)}{2}}}  \frac{p^{ng(1-t)}E(f_x,\eps,t)\Ll^{(N_1)}(1-t,\mathrm{St}(f),\chi^{-1}\eps^{-1}\omega^t) }{G(\chi'\eps\omega^{-t})^g {\left(p^{\frac{g(g+1)}{2}-gk}\alpha_p(f_x)^{2}\right)}^n{\lla  f_x\left\vert \left(  \begin{array}{cc}
0 & -1 \\
NS & 0                                                                                                                                                      \end{array}
\right)   \right., f_x \rra_{NS}}}. 
\end{align*} 
Moreover we have a function $L_p^*(x)$ such that the following equality of locally analytic functions around \'etale points $x \in \m F^{\mr{rig}}$ holds:
\begin{align}\label{Factorization}
L_p(x,[k-g])=E_1(f_x,{1},k-g)L_p^*(x).
\end{align}
\end{theo}
\begin{rem}
Let us denote by $\rho_{f,\mr{Sta}}$ the standard Galois representation constructed in \cite{ScholzeTor} and by $\mathcal{D}_{\mr{St}}(\rho_{f,\mr{Sta}})$ the semi-stable $(\varphi,N)$-module associated with it.  The factor 
\begin{align*}
\frac{p^{ng(1-t)}}{G(\eps\omega^{-t})^g {\left(p^{\frac{g(g+1)}{2}-gk}\alpha_p(f_x)^{2}\right)}^n}
\end{align*}
is the $\epsilon$-factor of the Weil--Deligne representation associated with the $(\varphi,N)$-submodule spanned by the eigenvectors of eigevanlues $\beta_i(f_x)$, as predicted in \cite[\S 6]{CoaMot}. Also the factor $E(f_x,\eps,t)$ is the one predicted by Coates. Note also that their product does not depend on the monodromy $N$.
\end{rem}
\begin{rem}
The key to obtain the factorization above is that the factor $E_1$ which brings the trivial zero for forms $\G_0(p)$-Steinberg is an analytic function of $x$: 
\begin{align*}
E_1(f_x,{1},k-g)=\prod_{i=1}^g{(1-\mathds{B}_i^{-1}(x)p^{g-i})}.
\end{align*}
\end{rem}
\begin{proof}
Let $f_x$ be the evaluation of $\boldF$ at $x$. We need to calculate the coefficient at $f_x \otimes f_x$ of $\frac{1}{{(-1)}^{gs} c^{s}_{t+1} A }  e_{g,g}^{\mr{ord}}   \h'_{1,t,\chi\eps\omega^{-t}}(z,w)$. 
We begin from the coefficient of $\mr{Id}{\otimes}T_{N^2R^2/NR_0} e_{g,g}^{\mr{ord}}  \h'_{1,t,\chi\eps\omega^{-t}}$ which is
\begin{align*}
\frac{\lla  f_x \left\vert \left(  \begin{array}{cc}
0 & -1 \\
NR_0p & 0                                                                                                                                                      \end{array}
\right)   \right. ,  \lla f_x \left\vert \left(  \begin{array}{cc}
0 & -1 \\
NR_0p & 0                                                                                                                                                      \end{array}
\right)   \right., (T_{N^2R^2/NR_0} \otimes \mr{Id}) e_{g,g}^{\mr{ord}}   \h'_{1,t,\chi\eps\omega^{-t}} \rra_{NR_0p} \rra_{NR_0p}}{\lla  f_x\left\vert \left(  \begin{array}{cc}
0 & -1 \\
NR_0p & 0                                                                                                                                                      \end{array}
\right)   \right., f_x \rra_{NR_0p}^2},
\end{align*}
as the Hecke operators are self-dual for the normalized Petersson product $\lla  f_x\left\vert \left(  \begin{array}{cc}
0 & -1 \\
N & 0                                                                                                                                                      \end{array}
\right)   \right., f_x \rra$.
The proof of \cite[Proposition 5.3]{PilHida} tells us that the $U_p$ operator on the right-hand-side  can be written as $p^{m\frac{gk-g(g+1)}{2}}\left[\G_0(Np^{m})\left( \begin{array}{cc}
                           1_g & 0 \\
                           0 & p^{m} 1_g
                          \end{array}
\right)\G_0(Np)\right]$. We know the relation 
\begin{align*}
\left\vert_k \left(
\begin{array}{cc}
0 & -1 \\
Np & 0                                                                                                                                                      \end{array}
\right)   \right.
\left\vert_k \left(
\begin{array}{cc}
p^m & 0 \\
0 & 1                                                                                                                                                     \end{array}
\right) \right.   = & \left\vert_k \left(
\begin{array}{cc}
0 & -1 \\
Np^{m+1} & 0                                                                                                                                                      \end{array}
\right)   \right. ,\\
\left\vert_k \left(
\begin{array}{cc}
0 & -1 \\
N & 0                                                                                                                                                      \end{array}
\right)   \right.
\left\vert_k \left(
\begin{array}{cc}
1 & 0 \\
0 & N/L                                                                                                                                                     \end{array}
\right) \right.   = &  \left\vert_k \frac{N}{L}\left(
\begin{array}{cc}
0 & -1 \\
L & 0                                                                                                                                                      \end{array}
\right)   \right..
\end{align*}
We use Lemma \ref{lemmaprojinfty} and  \cite[Lemma 4.1]{BS} to see that the numerator is 
\begin{align*}
\alpha_p(f_x)^{-2n} p^{n(gk-g(g+1) )} \lla  f_x \left\vert\tau_{NS}  \right. ,  \lla f_x \left\vert \tau_{N^2R^2p^{2n+1}}  \right.,  \h'_{1,t,\chi\eps\omega^{-t}} \rra_{N^2R^2p^{2n}} \rra_{NR_0p^{2n}},
\end{align*}
where we recall $S=R_0p$. 

From Proposition \ref{InteExpr} we have a term $\frac{\Omega_{k,s}({\frac{1}{2}-t}) p_{\frac{1}{2}-t}(t+g)}{ d_{\frac{1}{2}-t}(t+g)}$ appearing; we are left to evaluate
\begin{align*}
&\frac{\Omega_{k,s}({\frac{1}{2}-t}) p_{\frac{1}{2}-t}(t+g)}{B_{2g}(t){(-1)}^{gs}c^{s}_{g,t+g} d_{\frac{1}{2}-t}(t+g)} \\
= & \frac{\Omega_{k,s}({\frac{1}{2}-t})}{B_{2g}(t){(-1)}^{gs}c^{s}_{g,g+\frac{1}{2}} } \\
= &\frac{{(-1)}^{g(g+t+s)}\G_{2g}\left( g + \frac{1}{2} \right)}{  {2^{g(1+ 2t)}} \pi^{g+2g^2}}
\prod_{j=1}^s\frac{\G_g(s+\frac{g}{2}-\frac{j}{2})}{\G_g(s+\frac{g}{2}+1-\frac{j}{2})}\\
&\times {(-1)}^{\frac{gk}{2}} 2^{1+\frac{g(g+1)}{2}-g+2gt} \pi^{\frac{g(g+1)}{2}} \frac{\G_g(k-t+\frac{1-g}{2})\G_g(k-t-\frac{g}{2})}{\G_g(k-s+\frac{1}{2}-t)\G_g(k-s+\frac{1-g}{2}-t)}\\
= &{(-1)}^{\frac{3gk}{2}}2^{1+\frac{g(g+1)}{2}-2g}\pi^{-\frac{g(3g+1)}{2}} \frac{\G_{2g}(g+\frac{1}{2})}{\G_g(g+\frac{1}{2})\G_g(\frac{g+1}{2})}{\G_g(s+\frac{g-s}{2}+\frac{1}{2})\G_g(s+\frac{g-s}{2})}
\end{align*}
and to conclude we use that: 
\begin{align*}
\G_g(s)\G_g\left(s+\frac{1}{2}\right)& = \pi^{\frac{g(g-1)}{2}}2^{\frac{g(g+1)}{2}-2gs}\pi^{\frac{g}{2}}\prod_{i=1}^{g}\G(2s-i+1),\\
  \frac{\G_{2g}(g+\frac{1}{2})}{\G_g(g+\frac{1}{2})\G_g(\frac{g+1}{2})} & = \pi^{\frac{g^2}{2}}.
\end{align*}

The proof of second part of the theorem is very similar.  Comparing \cite[(7.13), (7.13)']{BS} we see that the only difference with the previous calculation is that we have to remove the factor 
\begin{align*}
p^{-\frac{g(g+1)}{2}} \prod_{i=1}^{g}(1-\beta_i^{-1}(f_x)\chi^{-1}(p)p^{-t}).
\end{align*} 
Here the power of $p$ compensate the missing power of $p$ in the term $A^*$ which appears in the interpolation formula (\ref{EisMeasure*}).
\end{proof}
\section{A formula for the derivative}
We now fix a form $f$ of weight $g+1$ and we suppose that $f$ has a Satake parameter, let us say  $\beta_g$, equal to $p^{-1}$. Conjecturally, this should imply that $\pi_f$ has not spherical level at $p$ (as otherwise the $\beta_i$'s should all be Weil number of weight zero). In particular, if $\varphi=(\rho,N)$ is the $L$-parameter (with values in $\mr{GSpin}_{2g+1}$) of $\pi_{f,p}$ then $N$ should have a $1$ in the $g,g+1$ entry. 
\begin{defin}
We say that $f$ is $\G_0(p)$-Steinberg at $p$ if $\beta_g=p^{-1}$ and the $g,g+1$ entry of $N$  is not zero. 
\end{defin}
This condition should conjecturally ensure us that the trivial zero is brought by $E_1$ and not by the missing factor $(1-p^{-s})$ of the completed $L$-function for $\mr{St}(f)$.
 We remark that $\G_0(p)$-Steinberg points are {\it isolated} in $\mr{Spec}(\m T)$ in the sense that only finitely many forms in $\mr{Spec}(\m T)$ satisfy this condition (as $p^{k-g}\beta_g$ must have fixed $p$-adic valuation). 
We know recall the main theorem of the paper:
\begin{theo}
Let $f$ be a Siegel form of weight $g+1$ and trivial Nebentypus; suppose that $f$ is $\G_0(p)$-Steinberg and the corresponding point on $\m F^{\mr{rig}}$ is \'etale, then 
\begin{align*}
\frac{\textup{d} }{\textup{d}s}{L_p(\mr{St}(f),s)}_{\vert_{s=0}}=\ell^{\mathrm{al}}(\mathrm{St}(f))E^*(f) \frac{\Ll^{(Np)}(\mr{St}(f),0)}{\Omega(f)},
\end{align*}
for $\Omega(f)$ a suitable complex period from \ref{BigTheo} and 
\begin{align*}
E^*(f)= & \frac{\prod_{i=2}^g{(1-\beta_i^{-1}p^{-1})}}{\prod_{i=1}^g(1-\beta_ip)}.
\end{align*}
\end{theo}
\begin{proof}
By hypothesis we have that $f$ corresponds to a point $x$ which is \'etale over $\W=\mr{Spf}(\Lambda)^{\mr{rig}}$. Let us write $t_0 = (z \mapsto \omega^{-g-1}(z)z^{g+1} )$; $t_0$ is a local uniformizer in $\mathcal{O}_{\W,[g+1]}$ and, by \'etalness, $t_0$ is also a local uniformizer for $\mathcal{O}_{\m F^{\mr {rig}},x}$. We have an isomorphism between the tangent spaces and this induces an isomorphism on derivations:
\begin{align*}\
\mathrm{Der}_{K}(\mathcal{O}_{\W,[g+1]},\C_p) \cong \mathrm{Der}_{K}(\mathcal{O}_{\m F^{\mr {rig}},x},\C_p).
\end{align*}
The isomorphism is made explict by fixing a common basis $\frac{\partial }{\partial T}$.
We consider the two variable $p$-adic $L$-function of Theorem \ref{BigTheo} with $\chi=1$ and $N_1=1$. We shall now write $L_p(k,t)$ for $L_p(x,t)$ (here $k=w(x)$).
 
If $w(x)$ is big enough, we know that $f_x$ must be a classical form of level prime to $p$ \cite[Th\'eor\`eme 1.1 (6)]{PilHida}. Then we know that the Euler factor at $p$ of $L(\mathrm{St}(f),s)$ is \begin{align*}
(1-\phi(p)p^{-s})\prod_{i=1}^g(1-\phi(p)\beta_{p,i}^{- 1}p^{-s})(1-\phi(p)\beta_{p,i}p^{-s}).
\end{align*}
Hence we can rewrite \begin{align*}
E(f_x,t) \Ll^{(p)}(\mathrm{St}(f),1-t) = & E_1(f_x,t)(1-p^{t-1})\prod_{i=1}^g{(1-\beta^{-1}_i p^{t-1})}\Ll(\mathrm{St}(f),1-t). 
\end{align*}
 The means that the two variables $p$-adic $L$-function vanishes on $t=1$. In what follows we shall use $s$ as a variable rather than $t$; remember that $t = k-g-s$.\\
The following formula is a straightforward consequence of the vanishing along $s=k-g-1$:
\begin{align*}
\frac{\textup{d} }{\textup{d}s}{L_p(k,s)}_{\vert_{s=0,k=g+1}} =  - \frac{\textup{d} }{\textup{d}k}{L_p(k,s)}_{\vert_{s=0,k=g+1}}.
\end{align*}
From the factorization in \ref{BigTheo} we obtain 
\begin{align*}
\frac{\textup{d} }{\textup{d}k}{L_p(k,k-g)}_{\vert_{k=g+1}}= \frac{\textup{d} \mathds{B}_g(k)}{\textup{d}k}_{\vert_{k=g+1}}\prod_{i=1}^{g-1}{(1-\mathds{B}_i^{-1}(g+1)p^{g-i})} L_p^*(g+1).
\end{align*}
To conclude we use \cite[Theorem 1.3]{RosLinv}:
\begin{align*}
 \ell^{\mathrm{al}}(\mathrm{St}(f)) = & -\frac{\textup{d}\mathds{B}_g(k)}{\textup{d}k}_{\vert_{k=g+1}}.
\end{align*}

\end{proof}
\subsection{Some examples}\label{examples}
In this last section we want to present some examples of the automorphic representation of $\mr{GSp}_4(\Q_p)$ which could be associated with a Siegel modular form of level $\G_0(p)$ and check whether our theorem applies or not.\\

We now consider an automorphic representation $\pi$ which at $p$ is isomorphic to the one labelled IIIa in \cite[Table A.1]{RSLocal}.  We see that it admits exactly two vectors invariant for $\G_0(N)$  (\cite[Table A.15]{RSLocal}, where $\G_0(p)$ is called $\mr{Si}(\pfrak)$) and the corresponding Satake parameters are respectively $ \chi(p),p^{-1}$ and $\chi^{-1}(p),p^{-1}$; only one corresponds to a $U_p$-ordinary Siegel form. The monodromy of the corresponding Weil--Deligne representation is given by the following matrix \cite[Table A.7, \S A.7]{RSLocal}:
\begin{align*}
\left(
\begin{array}{ccccc}
0 &   &    &   &  \\
  & 0 & -1 &   &  \\
  &   &  0 & 1 &  \\
  &   &    & 0 &  \\
  &   &    &   & 0                                                                                                                                                      \end{array}
\right).
\end{align*}
From \cite[Table A.10]{RSLocal} we see that the factor $E_1$ vanishes and the main theorem applies. Note that the corresponding automorphic representation belongs to the family of twists of the Steinberg representation for $\mr{GL}_2$ which has non zero Plancherel measure; we can than use \cite[Theorem 5.7]{ShinPlan} to deduce that there exists a weight $3$ Siegel form with this local representation.\\

We now consider an automorphic representation $\pi$ which at $p$ is isomorphic to the one labelled IIa in \cite[Table A.1]{RSLocal}; we see that it admits exactly one vector  invariant for $\G_0(N)$. The monodromy of the corresponding Weil--Deligne representation is given by 
\begin{align*}
\left(
\begin{array}{ccccc}
0 & 1 &   &   &  \\
  & 0 &   &   &  \\
  &   & 0 &   &  \\
  &   &   & 0 & -1 \\
  &   &   &   & 0                                                                                                                                                      \end{array}
\right).
\end{align*}
We see that the trivial zero here comes from the factor $E_2$ and we can not deal with it.\\

In a third case, the one labelled IVb, the monodromy of the corresponding Weil--Deligne representation is given by the same matrix as IIIa; hence our theorem would apply  if a form $f$  with this local representation exists (the Plancherel measure of this representation is $0$ so the previous theorem does not apply).

\bibliographystyle{alpha}
\bibliography{Bibliografy}
\end{document}